\documentclass[twoside,12pt]{article}

\usepackage{amssymb}
\usepackage{mathrsfs}
\usepackage{tikz}
\usepackage{mathrsfs}
\usepackage{amssymb}
\usepackage{amsmath}
\usepackage{bm}
\usepackage{graphicx}
\usepackage{pst-poly}  

\newtheorem{theorem}{Theorem}[section]%
\newtheorem{lemma}[theorem]{Lemma}%
\newtheorem{corollary}[theorem]{Corollary}%
\newtheorem{conjecture}[theorem]{Conjecture}

\newenvironment{proof}[1][Proof]{\noindent\textit{#1: } }{\hfill\rule{1mm}{2mm}}

\makeatletter \@addtoreset{equation}{section} \makeatother

\begin{document}

  \title{On Seymour's Second Neighborhood Conjecture of $m$-free
  Digraphs
 \thanks{Supported by NNSF of China (No. 11571044).}
  }
 \author{Hao Liang\thanks{Corresponding author: lianghao@mail.ustc.edu.cn}\,\,
 \\
  {\small Department of Mathematics}\\
  {\small Southwestern University of Finance and Economics}\\
  {\small Chengdu {\rm 611130}, China}\\
  \\
  Jun-Ming Xu \\
  {\small School of Mathematical Sciences}\\
  {\small University of Science and Technology of China}\\
 {\small Wentsun Wu Key Laboratory of CAS}\\
  {\small Hefei {\rm 230026}, China}\\
 }
\date{}

\maketitle {\centerline{\bf\sc Abstract}\vskip 8pt  This paper gives an approximate result
related to Seymour's Second Neighborhood conjecture, that is, for
any $m$-free digraph $G$, there exists a vertex $v\in V(G)$ and a
real number $\lambda_m$ such that $d^{++}(v)\geq \lambda_m d^+(v)$,
and $\lambda_m \rightarrow 1$ while $m \rightarrow +\infty$. This
result generalizes and improves some known results in a sense.



\section{Introduction}

Throughout this article, all digraphs are finite, simple and
digonless. As usual, for a vertex $v$ of the digraph $G$, we denote
by $N^+_G(v)$ the set of out-neighbors of $v$, $N^{++}_G(v)$ the set
of vertices at distance 2 from $v$. Let $d_G^+(v)=|N_G^+(v)|$ (the
out-degree of $v$) and $d_G^{++}(v)=|N_G^{++}(v)|$. We will omit the
subscript if the digraph is clear from the context.

In 1990, Seymour~\cite{DL95} proposed the following conjecture.

\begin{conjecture}\label{cjt1.1}\textnormal{(Seymour's Second Neighborhood Conjecture)}
For any digraph $G$, there exists a vertex $v$ in $G$ such that
$d^{++}(v)\geq d^+(v)$.
\end{conjecture}

We call the vertex $v$ in Conjecture~\ref{cjt1.1} a {\it Seymour
vertex}. In 2001, Kaneko and Locke~\cite{KL01} showed that any
digraph with the minimum outdegree less than 7 has a Seymour vertex.
In 2007, Fisher~\cite{F07} showed that any tournament has a Seymour
vertex; Fidler and Yuster~\cite{FY07} proved that any tournament
minus a star or a sub-tournament, and any digraph $G$ with minimum
degree $|V(G)|-2$ have Seymour vertices. {  In 2008,
Hamidoune~\cite{H08} proved that any vertex-transitive digraph has a
Seymour vertex. In 2013, Llad\'{o}~\cite{L13} proved that any
digraph with large connectivity has a Seymour vertex. In 2016, Cohn
{\it et al.}~\cite{CGHZ16} gave a probabilistic statement about
Seymour's conjecture and proved that almost surely there are a large
number of Seymour vertices in random tournaments and even more in
general random digraphs}. For a general digraph,
Conjecture~\ref{cjt1.1} is still open.

Another approach to Conjecture~\ref{cjt1.1} is to determinate the
maximum value of $\lambda$ such that there is a vertex $v$ in $G$
satisfying $d^{++}(v)\geq \lambda\, d^+(v)$ for any digraph $G$. In
2003, Chen, Shen and Yuster~\cite{CSY03} gave
$\lambda=0.657298\cdots$, which is the unique real root of the
polynomial $2x^3+x^2-1$. Furthermore, they improved this bound to
$0.67815\cdots$  mentioned in the end of the
article~\cite{CSY03}.

A digraph $G$ is called to be {\it $m$-free} if $G$ contains no
directed cycles of $G$ with length at most $m$. In 2010, Zhang and
Zhou~\cite{ZZ10} showed that for any $3$-free digraph $G$, there
exists a vertex $v$ in $G$ such that $d^{++}(v)\geq \lambda\,
d^+(v)$, where $\lambda= 0.6751\cdots$ { is the only real root in
the interval $(0,1)$ }of the polynomial $x^3+3x^2-x-1$. In this
paper, we consider general $m$-free digraphs and obtain the
following result.

\begin{theorem}\label{thm1.2}
{ Let $m$ be an arbitrarily fixed integer with $m\geq 3$ and $G$ be
an $m$-free digraph, then there exists a vertex $v$ in $G$ such that
$d^{++}(v)\geq \lambda_m d^+(v)$, where $\lambda_m$ is the only real
root in the interval $(0,1)$ of the polynomial}
  \begin{equation}\label{e1.1}
{  g_m(x)=2x^3-(m-3)x^2+(2m-4)x-(m-1)}.
 \end{equation}
Furthermore, $\lambda_m$ is increasing with $m$, and
$\lambda_m\rightarrow 1$ while $m\rightarrow +\infty$.

\end{theorem}

Since $G$ is simple and digonless, $G$ is $2$-free. When $m=2$, the
polynomial defined in (\ref{e1.1}) is exactly $2x^3+x^2-1$, and our
result can be considered to be a generalization of Chen {\it et
al.}'s result. When $m=3$, $\lambda_3= 0.6823\cdots$, which improves
Zhang {\it et al.}'s value on $\lambda_3$. When $m=4$, $\lambda_4=
0.7007\cdots$. From Theorem~\ref{thm1.2}, we immediately get the
following corollary.

\begin{corollary}\label{crl1.3}
{ For every $\varepsilon >0$, there is a positive integer $m$ such
that every $m$-free digraph contains a vertex $v$ with
$d^{++}(v)\geq (1-\varepsilon)\, d^+(v)$.}
\end{corollary}

{
 The first conclusion in Theorem~\ref{thm1.2} is our
main result. The proof proceeds by induction on the number of
vertices. In the induction step, we assume to the contrary that
$d^{++}(v)<\lambda_m d^+(v)$ for any vertex $v$ in $G$, where
$\lambda_m$ is the unique real root of $g_m(x)$ in the interval
$(0,1)$. Then we show that the assumption leads to a contradiction.
To this end, we need the following lemmas.

\begin{lemma}\label{lem1.4}
For $m\geq 3$, the polynomial $g_m(x)$ defined in (\ref{e1.1}) is
strictly increasing and has a unique real root in the interval
$(0,1)$.
\end{lemma}
\begin{proof}
Since
 $
 g_m(x)=2x^3-(m-3)x^2+(2m-4)x-(m-1),
 $
we have
 $$
 g_m'(x)=6x^2-2(m-3)x+(2m-4)=6x^2+2x+(2m-4)(1-x).
 $$
Clearly, $g_m'(x)>0$ when $m\geq 3$ and $x\in (0,1)$, which implies
$g_m(x)$ is strictly increasing in $[0,1]$. Since $g_m(0)=-m+1<0$
and $g_m(1)=2>0$, it follows that there is a unique real root in the
interval $(0,1)$ of the polynomial.
\end{proof}
 }


\begin{lemma}\textnormal{(Hamburger {\it et al.}~\cite{hhk07})}\label{lem1.5}
If one can delete $t$ edges from a digraph $G$ to make it acyclic,
then there exists a vertex $v$ in $G$ such that
$d^+(v)\leq\sqrt{2t}$.
\end{lemma}

\begin{lemma}\textnormal{(Liang and Xu~\cite{LX13})}\label{lem1.6}
If an $m$-free digraph $G$ is obtained from a tournament by deleting
$t$ edges, then one can delete from $G$ an additional $t/(m-2)$
edges so that the resulting digraph is acyclic.
\end{lemma}

Combining Lemma~\ref{lem1.5} with Lemma~\ref{lem1.6}, we can easily
get the following lemma.

\begin{lemma}\label{lem1.7}
If an $m$-free digraph $G$ is obtained from a tournament by deleting
$t$ edges, then there exists a vertex $v$ in $G$ such that
$d^+(v)\leq \sqrt{2t/(m-2)}$.
\end{lemma}

\begin{proof}
From Lemma~\ref{lem1.6}, an $m$-free $G$ is obtained from a
tournament by deleting $t$ edges, then we can delete $t/(m-2)$ edges
from $G$ to make it acyclic. From Lemma~\ref{lem1.5}, there exists a
vertex $v$ in $G$ such that $d^+(v)\leq \sqrt{2t/(m-2)}$.
\end{proof}

\vskip6pt

\section{{ Proof} of Theorem~\ref{thm1.2}}

We first prove the first conclusion { by induction on the number of
vertices}.
Theorem~\ref{thm1.2} is trivial for any digraph with 1 or 2
vertices. { Assume that Theorem~\ref{thm1.2} holds for all digraphs
with less than $n$ vertices. Let $G$ }be an $m$-free digraph with
$n$ vertices, $n\geq 3$ and $m\geq 3$. Assume to the contrary that
$d^{++}(v)<\lambda_m d^+(v)$ for any vertex $v$ in $G$, where
$\lambda_m$ { is the unique real root of $g_m(x)$ in the interval
$(0,1)$}. Our purpose is { to show that the assumption leads to a
contradiction}.

Let $u$ be a vertex in $G$ with minimum out-degree. Let $A=N^+(u)$,
$B=N^{++}(u)$, $a=|A|$ and $b=|B|$. By our assumption, we have
\begin{equation}\label{e2.1}
 b=d^{++}(u)< \lambda_m d^+(u)=\lambda_m a.
 \end{equation}

For any two disjoint subsets $X, Y\subseteq V(G)$, let $E(X, Y)$
denote the edges from $X$ to $Y$ and $e(X,Y)=|E(X,Y)|$. Since $G$ is
simple and digonless, we have that
 $$
 e(X,Y)+e(Y,X)\leq |X|\cdot|Y|.
 $$

For simplicity, for any subset $S\subseteq V(G)$, use $S$ to denote
the subgraph of $G$ induced by $S$. By the definitions of $A$ and
$B$, we have
\begin{equation}\label{e2.2}
\sum\limits_{v\in A} d_G^+(v)=|E(A)|+e(A,B).
 \end{equation}

By the choice of $u$, $d^+(v)\geq d^+(u)=a$ for any $v\in V(G)$, and
so
\begin{equation}\label{e2.3}
\sum\limits_{v\in A} d_G^+(v)\geq |A|\cdot d^+(u)=a^2.
 \end{equation}

Since $|E(A)|\leq a(a-1)/2$, we have
 $$
 e(A,B)=\sum\limits_{v\in A} d_G^+(v)-|E(A)|\geq a^2-a(a-1)/2= a(a+1)/2.
 $$
It follows that there exists $v\in A$ such that $e(v,B)\geq
e(A,B)/a\geq (a+1)/2$. Since $b=|B|\geq e(v,B)$ for any $v\in A$, it
follows that $\lambda_m a>b\geq e(v,B)\geq (a+1)/2>a/2$, which
implies
 \begin{equation}\label{e2.4}
\lambda_m > 1/2.
 \end{equation}

The subgraph $A$ can be obtained from a tournament of order $a$ by
deleting $t$ edges. Let $\theta =t/a^2$. Since $0\leq t\leq
a(a-1)/2$, we have { $0\leq \theta\leq(a-1)/2a <1/2$} and
\begin{equation}\label{e2.5}
|E(A)|=a(a-1)/2-t=(1/2-\theta)a^2-a/2<(1/2-\theta)\,a^2.
 \end{equation}
Combining (\ref{e2.2}), (\ref{e2.3}) with (\ref{e2.5}), we have that
\begin{equation}\label{e2.6}
e(A,B)=\sum\limits_{v\in A}
d_G^+(v)-|E(A)|>a^2-(1/2-\theta)a^2=(1/2+\theta)\,a^2.
 \end{equation}

Since $G$ is $m$-free, it follows that the subgraph $A$ is $m$-free.
From Lemma~\ref{lem1.7}, { there is a vertex $w_0\in A$} such that
\begin{equation}\label{e2.7}
{ d_A^+(w_0)}\leq \sqrt{2t/(m-2)}=a\sqrt{2\theta/(m-2)}.
 \end{equation}

{ Let $d_B^+(w_0)=|N_B^+(w_0)|$, then $d_B^+(w_0)\leq |B|=b$. Since
$d_A^+(w_0)+d_B^+(w_0)=d_G^+(w_0)$, it follows from (\ref{e2.1})
that $d_A^+(w_0)=d_G^+(w_0)-d_B^+(w_0)\geq d_G^+(w_0)-b\geq
a-\lambda_m a=(1-\lambda_m)\,a$}, that is,
\begin{equation}\label{e2.8}
{ d_A^+(w_0)}\geq (1-\lambda_m)\,a.
 \end{equation}
Combining (\ref{e2.7}) with (\ref{e2.8}), we have
$\sqrt{2\theta/(m-2)}a>(1-\lambda_m)\,a$, that is,
\begin{equation}\label{e2.9}
\theta>(m-2)(1-\lambda_m)^2/2.
 \end{equation}

Since $A$ is $m$-free and $|A|=a<n$, by induction hypothesis there
is a vertex { $w_1\in A$} such that { $|N_A^{++}(w_1)|\geq
\lambda_m|N_A^+(w_1)|$}, where $\lambda_m$ { is the unique real root
of $g_m(x)$ in the interval $(0,1)$}.

Let { $X=N_A^+(w_1)$, $Y=N_B^+(w_1)$} and $|Y|=d$. It follows from
(\ref{e2.1}) that
  \begin{equation}\label{e2.10}
  d=|Y|\leq |B|=b<\lambda_m\, a.
 \end{equation}

By the induction hypothesis,$|A-X|\geq |{ N_A^{++}(w_1)}|\geq
\lambda_m|X|$, that is, $(1+\lambda_m)|X|\leq |A|=a$. By
(\ref{e2.4}) $\lambda_m>\frac{1}{2}$, we have
 $$
 |X|\leq \frac{a}{1+\lambda_m}<\frac{2a}{3}.
 $$
By the choice of $u$, we have ${ d_G^+(w_1)}\geq d_G^+(u)=a$, and so
 \begin{equation}\label{e2.11}
d=|Y|=|{ N_G^+(w_1)}|-|X|>a-\frac{2a}{3}=\frac{a}{3}.
 \end{equation}
Combining (\ref{e2.10}) with (\ref{e2.11}), we have
\begin{equation}\label{e2.12}
a/3<d<\lambda_m\, a.
\end{equation}

For any $y\in Y$, use $d_{V-A-Y}^+(y)$ to denote the number of
out-neighbors of $y$ in $G$ not in $A\cup Y$. Since {
$d_G^{++}(w_1)< \lambda_m d_G^+(w_1)$} and { $d_A^{++}(w_1)\geq
\lambda_m d_A^+(w_1)$}, we have
 $$
 d_{V-A-Y}^+(y)\leq { d_G^{++}(w_1)-d_A^{++}(w_1)< \lambda_m d_G^+(w_1)-\lambda_m d_A^+(w_1)}=\lambda_m\, d.
 $$

Noting that $d_G^+(y)\geq d_G^+(u)=a$ and $\sum\limits_{y\in Y}
d_Y^+(y)=|E(Y)|\leq d(d-1)/2$, we obtain
 $$
 \begin{array}{ll}
 e(Y,A)&=\sum\limits_{y\in Y} |{ N_A^+(y)}|\\
        &\geq \sum\limits_{y\in Y} (a-d_{V-A-Y}^+(y)-d_Y^+(y))\\
        &> (a-\lambda_m d)\,d-\sum\limits_{y\in Y} d_Y^+(y)\\
        &> (a-\lambda_m d)\,d-d(d-1)/2\\
        &> (a-\lambda_m d-d/2)\,d,\\
\end{array}
 $$
that is
\begin{equation}\label{e2.13}
e(Y,A)>(a-\lambda_m d-d/2)d.
 \end{equation}

Combining(\ref{e2.1}), (\ref{e2.6}), (\ref{e2.9}) with
(\ref{e2.13}), we have
$$
 \begin{array}{ll}
 \lambda_m a^2&\geq ab\\
        &\geq e(A,B)+e(B,A)\\
        &\geq e(A,B)+e(Y,A)\\
        &> (1/2+\theta)\,a^2+(a-\lambda_m d-d/2)\,d\\
        &{ >} [1/2+(m-2)(1-\lambda_m)^2/2]\,a^2+(a-\lambda_m d-d/2)\,d\\
        &=-(\lambda_m+1/2)d^2+ad+[1/2+(m-2)(1-\lambda_m)^2/2]\,a^2,
\end{array}
 $$
 that is,
  \begin{equation}\label{e2.14}
\lambda_m
a^2>-(\lambda_m+1/2)\,d^2+ad+[1/2+(m-2)(1-\lambda_m)^2/2]\,a^2,
 \end{equation}
where $a/3<d<\lambda_m a$ (see (\ref{e2.12})). For $a/3\leq z\leq
\lambda_m\, a$, let the function
 $$
 f(z)=-(\lambda_m+1/2)z^2+az+[1/2+(m-2)(1-\lambda_m)^2/2]\,a^2.
  $$

Since $f(z)$ is a quadratic function with a negative leading
coefficient, the following inequality holds.
 \begin{equation}\label{e2.15}
 f(z)\geq \min\{f(a/3), f(\lambda_m a)\}\ \text{for any}\ z\in[a/3, \lambda_m a].
 \end{equation}
Combining (\ref{e2.14}) with (\ref{e2.15}), we have
 \begin{equation}\label{e2.16}
 \lambda_m a^2> f(d)\geq \min\{f(a/3), f(\lambda_m a)\}.
 \end{equation}

We first note that, since
 $$
 f(\lambda_m
 a)=\frac{a^2[-2\lambda_m^3+(m-3)\lambda_m^2-(2m-6)\lambda_m+(m-1)]}{2},
 $$
if $\lambda_m a^2>f(\lambda_m a)$, then
 $$
 \lambda_m
 a^2>\frac{a^2[-2\lambda_m^3+(m-3)\lambda_m^2-(2m-6)\lambda_m+(m-1)]}{2},
 $$
that is
 $$
 g_m(\lambda_m)=2\lambda_m^3-(m-3)\lambda_m^2+(2m-4)\lambda_m-(m-1)>0.
 $$
This fact shows that $\lambda_m$ { is not a root of the polynomial
$g_m(x)$, which contradicts our assumption} on $\lambda_m$.

It follows that $\lambda_m a^2\leq f(\lambda_m a)$, and so
$\lambda_m a^2>f(a/3)$ by (\ref{e2.16}). Since
 $$
 f(a/3)=\frac{a^2[9(m-2)\lambda_m^2-(18m-34)\lambda_m+(9m-4)]}{18}.
 $$
we have
 $$
 \lambda_m a^2>\frac{a^2[9(m-2)\lambda_m^2-(18m-34)\lambda_m+(9m-4)]}{18}.
 $$
Simplifying this inequality, we obtain
 $$
 9(m-2)\lambda_m^2-(18m-16)\lambda_m+(9m-4)<0.
 $$
This implies
\begin{equation}\label{e2.17}
\lambda_m>\frac{9m-8-\sqrt{54m-8}}{9(m-2)}.
 \end{equation}

{ Now we show (\ref{e2.17}) is a contradiction to that $\lambda_m$
is the only root in the interval $(0,1)$ of the polynomial
$g_m(x)$.} We rewrite the polynomial $g_m(x)$ as
\begin{equation}\label{e2.18}
g_m(x)=\frac 19(p(x)-q(x)),
 \end{equation}
where
 $$
 \begin{array}{rl}
 &p(x)=18x^3+9x^2-20x+5,\\
 &q(x)=9(m-2)x^2-(18m-16)x+(9m-4).
 \end{array}
 $$

The polynomial $q(x)$ has a real root
 \begin{equation}\label{e2.19}
 \varphi_m=\frac{9m-8-\sqrt{54m-8}}{9(m-2)},
 \end{equation}
that is
\begin{equation}\label{e2.20}
q(\varphi_m)=0.
 \end{equation}
Comparing (\ref{e2.17}) with (\ref{e2.19}), we have
\begin{equation}\label{e2.21}
\lambda_m\geq \varphi_m \ \ {\rm for}\ m\geq 3.
 \end{equation}
Since
 $$
 \begin{array}{rl}
 \varphi_m &=1+\frac{10-\sqrt{54m-8}}{9(m-2)}\\
  &=1+\frac{108-54m}{9(m-2)(10+\sqrt{54m-8})}\\
  &=1-\frac{6}{10+\sqrt{54m-8}},
  \end{array}
 $$
it is easy to see that { $\varphi_m$ is strictly increasing with $m$
for $m\geq 3$}.
 Thus we have
\begin{equation}\label{e2.22}
\varphi_m\geq
\varphi_3=1-\frac{6}{10+\sqrt{154}}>1-\frac{3}{10}=\frac{7}{10}.
 \end{equation}

A simple calculation gives us that $p(x)$ is a strictly increasing
function for $x>\frac{7}{10}$ and $p(\frac{7}{10})=1.584>0$. Noting
that $g_m(x)$ is a strictly increasing function over the interval
$[0,1]$, and by (\ref{e2.18}), (\ref{e2.20}), (\ref{e2.21}),
(\ref{e2.22}), we have
$$
g_m(\lambda_m)> g_m(\varphi_m)
             =\frac{1}{9}[p(\varphi_m)-q(\varphi_m)]
             =\frac{1}{9}p(\varphi_m)\
             > \frac{1}{9}p(\frac{7}{10})
             > 0.
$$
This fact shows that $\lambda_m$ is not a root of the polynomial
$g_m(x)$, a contradiction to our assumption, and so the first
conclusion follows.

\vskip6pt

We now prove the second conclusion. Since
$g_m(x)=2x^3-(m-3)x^2+(2m-4)x-(m-1)$, $g_m(\lambda_m)=0$ and
 $$
 \begin{array}{ll}
 g_{m+1}(x)&=2x^3-(m-2)x^2+(2m-2)x-m\\
           &=2x^3-(m-3)x^2+(2m-4)x-(m-1)-x^2+2x-1\\
           &=g_m(x)-(1-x)^2,
\end{array}
 $$
for any $m\geq 3$ we have
 $$
 g_{m+1}(\lambda_m)=g_m(\lambda_m)-(1-\lambda_m)^2=-(1-\lambda_m)^2<0=g_{m+1}(\lambda_{m+1}).
 $$

Since { $g_m(x)$ is strictly increasing in the interval $(0,1)$ for
any $m\geq 3$} by Lemma~\ref{lem1.4}, it follows that
$\lambda_m<\lambda_{m+1}$, which implies that $\lambda_m$ is
increasing with $m$.

We rewrite $g_m(x)$ as
$$g_m(x)=2x(x^2-1)+2x^2-(m-1)(1-x)^2.$$

It is easy to check that
$\mu_m=\frac{\sqrt{m-1}}{\sqrt{m-1}+\sqrt{2}}\in (0,1)$ is a real
root of the polynomial $2x^2-(m-1)(1-x)^2$. It follows that
$g_m(\mu_m)=2\mu_m(\mu_m^2-1)<0=g_m(\lambda_m)$. Since $g_m(x)$ { is
strictly increasing in the interval $(0,1)$} by Lemma~\ref{lem1.4},
we have
$$0<\mu_m<\lambda_m<1.$$

Since $\lim\limits_{m \rightarrow +\infty} \mu_m=\lim\limits_{m
\rightarrow +\infty} \frac{\sqrt{m-1}}{\sqrt{m-1}+\sqrt{2}}=1$, it
follows that $\lim\limits_{m \rightarrow +\infty} \lambda_m=1$.

The proof of Theorem~\ref{thm1.2} is complete.


\end{document}